\documentclass[12pt,leqno]{article}
\usepackage{amsfonts}
\usepackage{amsmath}
\usepackage{mathrsfs}
\usepackage{amssymb}
\usepackage{amsmath,color}
\usepackage{amsthm}
\usepackage{amstext}
\usepackage{amsopn}
\usepackage{texdraw}
\usepackage{graphicx}
\usepackage{subfigure}
\usepackage{hyperref}
\usepackage{amsmath}
\usepackage{amsthm}
\usepackage{amstext}
\usepackage{amsopn}
\usepackage{graphicx}
\usepackage{cite}
\newtheorem{theorem}{Theorem}[section]
\newtheorem{lemma}[theorem]{Lemma}

\newtheorem{corollary}[theorem]{Corollary}
\newtheorem{conjecture}[theorem]{Conjecture}
\theoremstyle{definition}

\theoremstyle{remark}
\newtheorem{remark}[theorem]{Remark}
\numberwithin{equation}{section}

\newcommand{\ba}{\begin{array}}
\newcommand{\ea}{\end{array}}

\newcommand{\ds}{\displaystyle}
\begin{document}
\date{}

\newcommand{\pe}{\psi}
\def\d{\delta}
\def\ds{\displaystyle}
\def\e{{\epsilon}}
\def\eb{\bar{\eta}}
\def\enorm#1{\|#1\|_2}
\def\Fp{F^\prime}
\def\fishpack{{FISHPACK}}
\def\fortran{{FORTRAN}}
\def\gmres{{GMRES}}
\def\gmresm{{\rm GMRES($m$)}}
\def\Kc{{\cal K}}
\def\norm#1{\|#1\|}
\def\wb{{\bar w}}
\def\zb{{\bar z}}

\newcommand{\Pa}{\partial}


\def\bfE{\mbox{\boldmath$E$}}
\def\bfG{\mbox{\boldmath$G$}}

\title {\LARGE \bf Global Threshold Dynamics of a Stochastic Differential Equation SIS Model
 }

\author{{Chuang Xu\thanks{
E-mail address: cx1@ualberta.ca}}\\
\small{Department of Mathematics,}\\
\small{Harbin Institute of Technology, Harbin, Heilongjiang 150001,
P. R. China and}\\
\small{Department of Mathematical and Statistical Sciences,}\\
\small{University of Alberta, Edmonton, Alberta,T6G 2G1, Canada.}}
\maketitle
\begin{abstract}
In this paper, we further investigate the global dynamics of a stochastic differential equation SIS (Susceptible-Infected-Susceptible) epidemic model recently proposed in [A. Gray et al., SIAM. J. Appl. Math., 71 (2011), 876-902]. We present a  stochastic threshold theorem in term of a \textit{stochastic basic reproduction number} $R_0^S:$ the disease dies out with probability one if $R_0^S<1,$ and the disease is recurrent if $R_0^S\geqslant1.$ We prove the existence and global asymptotic stability of a unique invariant density for the Fokker-Planck equation associated with the SDE SIS model when $R_0^S>1.$ In term of the profile of the invariant density, we define a \textit{persistence basic reproduction number} $R_0^P$ and give a persistence threshold theorem: the disease dies out with large probability if $R_0^P\leqslant1,$ while persists with large probability if $R_0^P>1.$ Comparing the \textit{stochastic disease prevalence} with the \textit{deterministic disease prevalence}, we discover that the stochastic  prevalence is bigger than the deterministic  prevalence if the deterministic basic reproduction number $R_0^D>2.$ This shows that noise may increase severity of disease. Finally, we study the asymptotic dynamics of the stochastic SIS model as the noise vanishes and establish a sharp connection with the threshold dynamics of the deterministic SIS model in term of a \textit{Limit Stochastic Threshold Theorem}.

{\bf Keywords} SIS epidemic model, basic reproduction number, global threshold dynamics, invariant density, stochastic differential equation, Fokker-Planck equation.
\end{abstract}
\maketitle

\section{Introduction}\label{sect1}

Transmission of diseases are usually described in compartmental models \cite{DHM,HY,KM,VW}. Most of the early work on mathematical epidemiology focused on deterministic models \cite{AM,DH,DHM,VW}. However, stochastic noise plays an indispensable role in transmission of diseases, especially in a small total population. Therefore it seems more practical to consider stochastic epidemic models.

One way to describe the compartmental stochastic epidemic models is via stochastic differential equations (SDE), which change parameters involved in deterministic models randomly. This is the so-called \textit{parameter perturbation method} of deriving an SDE model from its deterministic counterpart \cite{DGM,G,TBV}.  Such environmental noise reflected in specific parameters of the stochastic epidemic models may indicate whether  severity of disease increases or decreases based on what the deterministic compartmental epidemic models predict and thus influence of these parameters is well captured.

It is well known that epidemic threshold theorem holds for most deterministic compartmental epidemic models \cite{KM,VW}: the disease free equilibrium (DFE) is globally asymptotically stable if the basic reproduction number $R_0\leqslant1,$ and an endemic equilibrium exists and is globally asymptotically stable if $R_0>1.$ However, there seems no global threshold result for SDE epidemic models in literature.

In \cite{TBV}, a stochastic differential equation SIR (susceptible-infected-recovered) model is discussed and a sufficient but not necessary condition for local almost sure asymptotic stability of DFE is presented. Later, in \cite{Lu}, an improved sufficient condition is given by considering an SIRS model which specializes the SIR model in \cite{TBV}. Both of these results are derived by constructing Lyapunov functions. In \cite{C}, a sufficient and necessary condition for local almost sure asymptotic stability of DFE is proved by considering the linearized system.

Perturbing the constant $\beta$ in the deterministic SIS type model
\begin{equation}
  \label{ode}
  \begin{split}
    \frac{{\rm d}S(t)}{{\rm d}t}=&\mu N-\beta S(t)I(t)+\gamma I(t)-\mu S(t),\\
    \frac{{\rm d}I(t)}{{\rm d}t}=&\beta S(t)I(t)-(\mu+\gamma)I(t),
  \end{split}
\end{equation}
randomly  by $\tilde{\beta}{\rm d}t=\beta {\rm d}t+\sigma {\rm d}B(t),$  Gray et al. \cite{G} proposed the following SDE SIS model
\begin{equation}
  \label{0}
  \begin{split}
    {\rm d}S(t)=&[\mu N-\beta S(t)I(t)+\gamma I(t)-\mu S(t)]{\rm d}t-\sigma S(t)I(t){\rm d}B(t),\\
    {\rm d}I(t)=&[\beta S(t)I(t)-(\mu+\gamma)I(t)]{\rm d}t+\sigma S(t)I(t){\rm d}B(t)
  \end{split}
\end{equation}
with initial values $S_0+I_0=N.$ Over a long period of time, the average total population size is assumed to be constant $N.$ Here  $S(t)$ and $I(t)$ are the number of susceptibles and the number of infectives at time $t,$ respectively; $B(t)$ is a Brownian motion; $\mu$ is the per capita (birth) death rate, $\gamma$ the recovery rate, $\beta$ the disease transmission coefficient, and $\sigma$ the standard deviation of the noise. For more details of the model, we refer the reader to \cite{G,HY}.

Since $S(t)+I(t)=N,$ which is a constant, it reduces to studying the following one-dimensional SDE:
\begin{equation}
  \label{1}
  {\rm d}I(t)=I(t)\Big([\beta N-\mu-\gamma-\beta I(t)]{\rm d}t+\sigma [N-I(t)]{\rm d}B(t)\Big).
\end{equation}

In \cite{G}, the global existence, uniqueness, boundedness and positiveness of \eqref{1} are proved, and the counterpart of the basic reproduction number in SDE SIS model is defined by $R_0^S:=\displaystyle\frac{\beta N}{\mu+\gamma}-\frac{\sigma^2N^2}{2(\mu+\gamma)}.$ When $R_0^S<1$ and $\sigma^2\leqslant\displaystyle\frac{\beta}{N},$ or $\sigma^2>\displaystyle\max\bigg\{\frac{\beta}{N},\frac{\beta^2}{2(\mu+\gamma)}\bigg\},$ it is shown that the disease dies out with probability one. When $R_0^S>1,$ it is demonstrated that the disease persists in the sense that \begin{equation}\label{p}
\liminf_{t\to\infty}I(t)\leqslant \tilde{I}_*(\sigma)\leqslant\limsup_{t\to\infty}I(t),
\end{equation} where $\tilde{I}_*(\sigma)=\displaystyle\frac{1}{\sigma^2}\Big(\sqrt{\beta^2-2\sigma^2(\mu+\gamma)}-(\beta-\sigma^2N)\Big).$ Moreover, a unique stationary distribution of model \eqref{1} is proved to exist and formulae for the mean and variance of the stationary distribution are presented. Numerical simulations reveal that as long as $R_0^S<1,$ the disease will die out in the long run. Thus a conjecture on the extinction of the disease is proposed in \cite{G}:
\vskip 0.2cm
\begin{conjecture}\label{co1}\hfill
  If $R_0^S<1$ and $\displaystyle\frac{\beta^2}{2(\mu+\gamma)}\geqslant\sigma^2>\frac{\beta}{N},$ then the disease will die out with probability one.
\end{conjecture}
\vskip 0.5\baselineskip

In this paper, we further investigate the dynamics of the stochastic SIS model \eqref{0}. Instead of using the standard approach of constructing a Lyapunov function, we use {\em Feller's test} for explosions (see Lemma~\ref{le3} in the Appendix) to establish a global threshold theorem for model \eqref{0}. We prove that the disease dies out with probability one if $R_0^S<1,$ and the disease is \textit{recurrent} in the sense that the process of disease transmission is recurrent if $R_0^S\geqslant1.$ For a precise definition of recurrence in this paper, we refer the reader to Lemma~\ref{le3} in the Appendix. Although the existence of the stationary distribution is presented in \cite{G}, the profile of the stationary distribution, which contains more information about the SDE SIS model, is not addressed. Instead of studying the SDE, we investigate the Fokker-Planck equation (FPE) associated with \eqref{1} and solve the invariant density. By constructing a proper Lyapunov function, we show that the invariant density is unique and globally asymptotically stable. We further define a \textit{persistence basic reproduction number} $R_0^P,$  \textit{disease-extinction with large probability} and \textit{disease-persistence with large probability} in term of the profile of the invariant density. We show that the invariant density peaks at $0$ if $R_0^P<1,$ while peaks at some positive number $I_*(\sigma)$ if $R_0^P>1.$ Comparing the \textit{stochastic disease prevalence} with the \textit{deterministic disease prevalence}, we discover that the stochastic  prevalence is greater than the deterministic  prevalence if the {\em deterministic basic reproduction number} $R_0^D>2.$ These results reveal that stochasticity favors disease extinction if the basic reproduction number is small, but enhances severity of disease if the basic reproduction number is large. 
By analyzing the asymptotic behavior of the invariant density as $\sigma\to0,$ we establish a sharp connection with model \eqref{ode} via a \textit{Limit Stochastic Threshold Theorem}.

This paper is organized as follows. In the next section, we present the global stochastic threshold theorem. We derive the existence, uniqueness, global stability and an explicit formula of an invariant density of the FPE associated with \eqref{1} in Section 3. Then  we define the persistence basic reproduction number and give a persistence threshold theorem in Section 4. Finally, we study the asymptotic behavior of the invariant density as the noise vanishes and establish the Limit Stochastic Threshold Theorem in Section 5. We list some preliminary results on Feller's test and FPE in the Appendix.
\section{Stochastic Threshold Theorem}
\label{sect2}
In this section, we prove a stochastic threshold theorem: the disease dies out with probability one if $R_0^S<1,$ but is recurrent in the sense that the process of disease transmission is recurrent if $R_0^S\geqslant1.$

For the reader's convenience, we first restate the global existence of a unique bounded positive solution of \eqref{1}.
\vskip 0.2cm
\begin{lemma}{\rm \cite{G}}\label{le1}
  For any given initial value $I(0)=I_0\in(0,N),$ the SDE \eqref{1} has a unique global positive solution $I(t)\in(0,N)$ for all $t\geqslant0$ with probability one, namely,
  $$\mathbb{P}\{I(t)\in(0,N) : \forall\ t\geqslant0\}=1.$$
\end{lemma}


Now we state the following stochastic threshold theorem, which in particular, verifies Conjecture \ref{co1} in Section \ref{sect1}.

\vskip 0.2cm
\begin{theorem}[\textbf{Stochastic Threshold Theorem}]\label{th1}\hfill
  \begin{enumerate}
    \item[{\rm (1)}] If $R_0^S<1,$ 
  for any given initial value $I(0)=I_0\in(0,N),$  $$\mathbb{P}\Big\{\lim_{t\to\infty}I(t)=0\Big\}=1.$$ In other words, the disease dies out with probability one.
  \item[{\rm (2)}] If $R_0^S\geqslant1,$ 
      for any given initial value $I(0)=I_0\in(0,N),$ $$\mathbb{P}\Big\{\underset{0\leqslant t<\infty}{\sup}I(t)=N\Big\}=\mathbb{P}\Big\{\underset{0\leqslant t<\infty}{\inf}I(t)=0\Big\}=1.$$ In particular, the process $I_t$ is recurrent: for every $x\in(0,N),$ we have $$\mathbb{P}\{I(t)=x:\ \exists\ t\in[0,\infty)\}=1.$$ In other words, the disease is recurrent.
  \end{enumerate}
  \end{theorem}
  \vskip 0.2cm
\begin{proof}
By Lemma \ref{le1}, let $Y(t)=g(I(t)),$ where $g:(0,N)\to\mathbb{R}$ is defined by
\begin{equation}
  \label{3}
  g(\xi)=\log\frac{\xi}{N-\xi}.
\end{equation} Using I\^{t}o's formula, it is easy to verify that $Y(t)$ solves the following SDE
\begin{equation}
  \label{2}
  {\rm d}Y(t)=\bigg[\Big(\beta N-\mu-\gamma-\frac{1}{2}\sigma^2N^2\Big)-(\mu+\gamma)e^{Y(t)}+\frac{\sigma^2N^2e^{Y(t)}}{1+e^{Y(t)}}\bigg]{\rm d}t+\sigma N{\rm d}B(t).
\end{equation}

In fact, by \eqref{1}, \[{\rm d}Y(t)=\bigg[b(I(t))g'(I(t))+\frac{1}{2}\Big(a\big(I(t)\big)\Big)^2g''(I(t))\bigg]{\rm d}t+a(I(t))g'(I(t)){\rm d}B(t),\] where $a(\xi)=\sigma(N-\xi)\xi,$ $b(\xi)=(\beta N-\mu-\gamma)\xi-\beta \xi^2,$ $g'(\xi)=\displaystyle\frac{1}{N-\xi}+\frac{1}{\xi}$ and\\ $g''(\xi)=\displaystyle\frac{1}{(N-\xi)^2}-\frac{1}{\xi^2}.$ Thus \begin{equation}
  \label{8}
  {\rm d}Y(t)=\bigg[\Big(\beta N-\mu-\gamma-\frac{1}{2}\sigma^2N^2\Big)-\frac{(\mu+\gamma)I(t)}{N-I(t)}+\sigma^2NI(t)\bigg]{\rm d}t+\sigma N{\rm d}B(t).
\end{equation} Substituting $I(t)=g^{-1}(Y(t))=\displaystyle\frac{Ne^{Y(t)}}{1+e^{Y(t)}}$ into \eqref{8}, we obtain \eqref{2}.

 Next, by Lemma \ref{le1} and Lemma \ref{le2} in the Appendix, we have the global existence of a unique solution to \eqref{2}:
\begin{theorem}
  \label{th10}
For any given initial value $Y(0)=Y_0\in\mathbb{R},$ the SDE \eqref{2} has a unique global solution $Y(t)\in\mathbb{R}$ for all $t\geqslant0$ with probability one, namely,
  \begin{equation}\label{22}\mathbb{P}\{Y(t)\in\mathbb{R}\ :\ \forall\ t\geqslant0\}=1.\end{equation}
\end{theorem}
By Theorem~\ref{th10}, the process $Y_t$ is {\em conservative} (for its definition, see p.153 in \cite{I}).

The scale function defined before Lemma~\ref{le3} in the Appendix for \eqref{2} is given by$$\psi(x)=\int_0^x\phi(\xi){\rm d}\xi$$with \[\begin{split}\phi(\xi)=&\exp\Bigg\{-\frac{2}{\sigma^2N^2}\int_0^{\xi}\bigg[\Big(\beta N-\mu-\gamma-\frac{1}{2}\sigma^2N^2\Big)-(\mu+\gamma)e^r+\frac{\sigma^2N^2e^r}{1+e^r}\bigg]{\rm d}r\Bigg\}\\
  =&\exp\Bigg\{-\frac{2}{\sigma^2N^2}\Big(\beta N-\mu-\gamma-\frac{1}{2}\sigma^2N^2\Big)\xi+\frac{2(\mu+\gamma)}{\sigma^2N^2}\big(e^{\xi}-1\big)-2\log\frac{e^{\xi}+1}{2}\Bigg\}.\end{split}\] It is obvious that $\psi(\infty)=\infty.$
Note that $$\phi(\xi)\thicksim\exp\bigg\{-\frac{2}{\sigma^2N^2}\Big(\beta N-\mu-\gamma-\frac{1}{2}\sigma^2N^2\Big)\xi-\frac{2(\mu+\gamma)}{\sigma^2N^2}+2\log2\bigg\},\ \mbox{as}\ \xi\to-\infty.$$ Recall that $R_0^S=\displaystyle\frac{\beta N}{\mu+\gamma}-\frac{\sigma^2N^2}{2(\mu+\gamma)}.$
If $R_0^S<1,$ then $\beta N-\mu-\gamma-\displaystyle\frac{1}{2}\sigma^2N^2<0,$ which implies that $\psi(-\infty)>-\infty;$ if $R_0^S\geqslant1,$ then $\beta N-\mu-\gamma-\displaystyle\frac{1}{2}\sigma^2N^2\geqslant0,$ which implies $\psi(-\infty)=-\infty.$
Hence by Lemma~\ref{le2} and Lemma~\ref{le3} in the Appendix, we arrive at the conclusions.
\end{proof}

\section{Invariant density}

In this section, for $R_0^S>1,$ we give the existence, uniqueness, global asymptotic stability and an explicit formula of the invariant density of the FPE associated with SDE \eqref{1}
\begin{equation}
  \label{10}
  \frac{\partial p(t,x)}{\partial t}=-\frac{\partial}{\partial x}\Big\{x\big[\beta N-\mu-\gamma-\beta x\big]p(t,x)\Big\}+\frac{1}{2}\sigma^2\frac{\partial^2}{\partial x^2}\big(x^2(N-x)^2p(t,x)\big).
\end{equation}
\vskip 0.2cm
\begin{theorem}\label{th11}
If $R_0^S>1,$ then there exists a unique invariant probability measure $\nu_{\sigma}^s$ for \eqref{10} which has the density $p_{\sigma}^s$ with respect to the Lebesgue measure. Moreover, the invariant density $p_{\sigma}^s$ is globally asymptotically stable in the sense that
  \begin{equation}
    \lim_{t\to\infty}\int_0^N|\mathcal{P}(t)q(x)-p_{\sigma}^s(x)|{\rm d}x=0,\ \forall\ q\in L^1_+((0,N)),
  \end{equation}where $\{\mathcal{P}(t)\}_{t\geqslant0}$ is the Markov semigroup defined by \eqref{10} and $L^1_+((0,N)):=\{w\in L^1(\mathbb{R}): \int_0^Nw(x){\rm d}x=1,\ w(x)=0\  \text{for}\ x\geqslant N\ \text{or}\ x\leqslant0, \mbox{and}\ w(x)\geqslant0\ \text{for}\ x\in\mathbb{R}\}.$
In addition, the process $I_t$ has the ergodic properties, i.e., for any $\nu_{\sigma}^s$-integrable function $K:$
\begin{equation}
  \mathbb{P}_{I_0}\bigg(\lim_{t\to\infty}\frac{1}{t}\int_0^tK(I_\tau){\rm d}\tau=\int_0^NK(y)\nu_{\sigma}^s({\rm d}y)\bigg)=1,
\end{equation}for all $I_0\in(0,N).$ Moreover, the unique invariant density $p_{\sigma}^s$ of the Markov semigroup $\{\mathcal{P}(t)\}_{t\geqslant0}$ is given by
\begin{equation}
  \label{17}p_{\sigma}^s(x):=CN^3\frac{x^{c_0(R_0^S-1)-1}}{(N-x)^{c_0(R_0^S-1)+3}}e^{-c_0\frac{x}{N-x}}
\end{equation}with
\begin{equation}\label{29}
C^{-1}=c_0^{-c_0(R_0^S-1)}\Big[(R_0^S)^2+c_0^{-1}(R_0^S-1)\Big]\Gamma(c_0(R_0^S-1)),\end{equation} $c_0=\displaystyle\frac{2(\mu+\gamma)}{\sigma^2N^2}$ and $\Gamma(\cdot)$ the gamma function.
\end{theorem}

To prove Theorem~\ref{th11}, we need to first study the following FPE associated with SDE \eqref{2}
\begin{equation}
  \label{5}
  \frac{\partial u(t,\xi)}{\partial t}=-\frac{\partial}{\partial \xi}\Bigg\{\bigg[(\beta N-\mu-\gamma-\frac{1}{2}\sigma^2N^2)-(\mu+\gamma)e^{\xi}+\frac{\sigma^2N^2e^{\xi}}{1+e^{\xi}}\bigg]u(t,\xi)\Bigg\}+\frac{1}{2}\sigma^2N^2\frac{\partial^2u(t,\xi)}{\partial \xi^2}.
\end{equation}
Denote by $\{\mathcal{U}(t)\}_{t\geqslant0}$ the Markov semigroup (also called {\em stochastic semigroup}, see Remark~11.8.1 on p.370 in \cite{LM}) defined by \eqref{5}.

Now we give the existence, uniqueness and asymptotic stability of an invariant density of the Markov semigroup $\{\mathcal{U}(t)\}_{t\geqslant0}.$
\vskip 0.2cm
\begin{theorem}
  \label{th3}
If $R_0^S>1,$ then there exists a unique invariant probability measure $\kappa_{\sigma}^s$ for \eqref{5} which has the density $u^s_{\sigma}$ with respect to the Lebesgue measure. Moreover, the invariant density $u^s_{\sigma}$ is globally asymptotically stable in the sense
  \begin{equation}
    \lim_{t\to\infty}\int_{-\infty}^{\infty}|\mathcal{U}(t)v(\xi)-u^s_{\sigma}(\xi)|{\rm d}\xi=0,\ \forall\ v\in L^1_+(\mathbb{R}),
  \end{equation} where $L^1_+(\mathbb{R}):=\{f\in L^1(\mathbb{R}): \int_{-\infty}^{\infty}f(x){\rm d}x=1\ \text{and}\ f(x)\geqslant0,\ \text{for}\ x\in\mathbb{R}\}.$
In addition, the process $Y_t$ has the ergodic properties, i.e., for any $\kappa_{\sigma}^s$-integrable function $F:$
\begin{equation}
  \mathbb{P}_{Y_0}\Big(\lim_{t\to\infty}\frac{1}{t}\int_0^tF(Y_\tau){\rm d}\tau=\int_{-\infty}^{\infty}F(\eta)\kappa_{\sigma}^s({\rm d}\eta)\Big)=1,
\end{equation}for all $Y_0\in\mathbb{R}.$ Moreover, the unique invariant density $u^s_{\sigma}$ of the Markov semigroup $\{\mathcal{U}(t)\}_{t\geqslant0}$ is given by
\begin{equation}
  \label{16}u^s_{\sigma}(\xi):=Ce^{c_0(R_0^S-1)\xi-c_0 e^{\xi}+2\ln(e^{\xi}+1)},
\end{equation}
where $C$ and $c_0$ are defined in \eqref{29} in Theorem~\ref{th11}.
\end{theorem}
\vskip 0.2cm
\begin{proof}
Since $Y_t$ is conservative and non-degenerate (i.e., $\displaystyle\frac{1}{2}\sigma^2N^2>0,\
\forall\ \xi\in\mathbb{R}$), by \cite{K} (see Chapter III), there exists a unique classical fundamental solution to \eqref{5} (see also p.153 in \cite{I}). Thus by Theorem~\ref{th0} in the Appendix, there exists a generalized solution $u(t,\xi)\in L^1_+(\mathbb{R})$ for all $t>0,\ \xi\in\mathbb{R}$ provided that the initial density $u_0\in L^1_+(\mathbb{R}).$

For simplicity, we denote $u^s_{\sigma}$ by $u^s$ throughout this proof.
We first give the existence, uniqueness and global asymptotic stability of $u^s.$  Let $V(\xi):=e^{-\alpha_0\xi}+\xi^2$ with $\alpha_0=\displaystyle\frac{\beta N-\mu-\gamma}{\sigma^2N^2}-\frac{1}{2}.$ Notice that $\alpha_0>0$ for $R_0^S>1.$ It is straightforward to verify that $V$ is a Lyapunov function defined in the Appendix. Let $\delta=2,$ $\alpha_1=c_1=1,$ $c_2=\displaystyle\frac{\bigg(\beta N-\mu-\gamma-\displaystyle\frac{1}{2}\sigma^2N^2\bigg)^2}{4\sigma^2N^2}.$ Then there exists sufficiently large $\alpha_2>0$ such that both inequalities in \eqref{18} in Theorem~\ref{th9} hold. Hence by Theorem~\ref{th9} in the Appendix, we have the uniqueness and global asymptotic stability of the invariant density $u^s.$ Next, we verify that $u^s$ given by \eqref{16} is an invariant density. From \eqref{5}, it suffices to show that \begin{equation}
  \label{20}-\frac{{\rm d}}{{\rm d}\xi}\Bigg\{\bigg[\Big(\beta N-\mu-\gamma-\frac{1}{2}\sigma^2N^2\Big)-(\mu+\gamma)e^{\xi}+\frac{\sigma^2N^2e^{\xi}}{1+e^{\xi}}\bigg]u^s(\xi)\Bigg\}+\frac{1}{2}\sigma^2N^2\frac{{\rm d}^2u^s(\xi)}{{\rm d} \xi^2}=0.
\end{equation} In fact, $u^s$ is a solution of the ODE
\begin{equation}\label{21}
  \frac{{\rm d}u}{{\rm d}\xi}=\frac{2\bigg[\Big(\beta N-\mu-\gamma-\displaystyle\frac{1}{2}\sigma^2N^2\Big)-(\mu+\gamma)e^{\xi}+\displaystyle\frac{\sigma^2N^2e^{\xi}}{1+e^{\xi}}\bigg]}{\sigma^2N^2}u,
\end{equation}which implies $u^s$ solves \eqref{20}. Note that $0<\Gamma(c_0(R_0^S-1))<\infty$ for $R_0^S>1,$ and thus $C$ defined in \eqref{29} is finite. Hence $u^s$ is an invariant density.
\end{proof}
By Theorem~\ref{th3} and Theorem~\ref{th8} in the Appendix, we prove Theorem~\ref{th11}.

\section{Persistence Threshold Theorem}

In this section, for $R_0^S>1,$ we define a persistence basic reproduction number, disease-extinction with large probability and disease-persistence with large probability in term of the profile of the invariant density $p^s_{\sigma}$ given by \eqref{17} and establish a persistence threshold theorem.

Let $R_0^P:=\displaystyle\frac{\beta N-\sigma^2N^2}{\mu+\gamma}$ be the \textit{persistence basic reproduction number}.
Model \eqref{1} is \textit{disease-persistent with large probability} if  the invariant density peaks at a  positive number (or equivalently, the mode of the stationary distribution is positive); otherwise, if the invariant density peaks at zero (or equivalently, the mode of the stationary distribution is zero), then model \eqref{1} is \textit{disease-extinct with large probability}.

In the following, we investigate the profile of the invariant density $p_{\sigma}^s.$
\vskip 0.2cm
\begin{theorem}[\textbf{Profile of Invariant Density}]\label{th5}
Assume $R_0^S>1.$ Then $$\lim_{x\uparrow N}p_{\sigma}^s(x)=0.$$
\begin{enumerate}
\item[{\rm (1)}.] Suppose $R_0^P<1.$ Then $$\lim_{x\downarrow0}p^s_{\sigma}(x)=\infty.$$ Moreover,
\begin{enumerate}
  \item[{\rm (a)}] if $\displaystyle R_0^P\leqslant\displaystyle\frac{4(\sqrt{c_0}-1)}{c_0},$ then  $p_{\sigma}^s$ is strictly decreasing in $(0,N);$
  \item[{\rm (b)}] if $R_0^P>\displaystyle\frac{4(\sqrt{c_0}-1)}{c_0},$ then $p_{\sigma}^s$ is strictly decreasing in $(0,I_-),$ $(I_+,N),$ and increasing in $(I_-,I_+),$ where \begin{equation}\label{+}
        I_{\pm}=\frac{N}{8}\bigg[(4-R_0^Pc_0)\pm\sqrt{(4+R_0^Pc_0)^2-16c_0}\bigg]
      \end{equation} and $c_0$ is defined in \eqref{29} in Theorem~\ref{th11}.\\
\end{enumerate}
\item[{\rm (2)}.] Suppose $R_0^P=1.$ Then  $$\lim_{x\downarrow0}p_{\sigma}^s(x)=C/N,$$ where $C$ is defined in \eqref{29} in Theorem~\ref{th11}. Moreover,
\begin{enumerate}
\item[{\rm (a)}] if $c_0\geqslant4,$ then  $p_{\sigma}^s$ is strictly decreasing in $(0,N);$
  \item[{\rm (b)}] if $c_0<4,$ then $p_{\sigma}^s$ is strictly increasing in $\Bigg(0,\bigg(1-\displaystyle\frac{c_0}{4}\bigg)N\Bigg)$ and decreasing in $\Bigg(\bigg(1-\displaystyle\frac{c_0}{4}\bigg)N,N\Bigg).$\\
\end{enumerate}
  \item[{\rm (3)}.] Suppose $R_0^P>1,$ then $$\lim_{x\downarrow0}p_{\sigma}^s(x)=0.$$ Moreover, $p_{\sigma}^s$ is strictly increasing in $(0,I_*({\sigma}))$ and decreasing in $(I_*({\sigma}),N),$ where \begin{equation}\label{-}
        I_*(\sigma)=\frac{N}{8}\bigg[(4-R_0^Pc_0)+\sqrt{(4+R_0^Pc_0)^2-16c_0}\bigg].
      \end{equation} In particular, the invariant density peaks at $I_*(\sigma).$
\end{enumerate}
\end{theorem}
\vskip 0.2cm
\begin{proof}
  Note that $$p_{\sigma}^s(x)=CN^3e^{h_{\sigma}(x)},$$ where $h_{\sigma}(x)=c_0(R_0^P-1)\ln x-[c_0(R_0^P-1)+4]\ln(N-x)-c_0\displaystyle\frac{x}{N-x}.$ The limit $\lim_{x\uparrow N}p_{\sigma}^s(x)=0$ follows from $\lim_{x\uparrow N}h_{\sigma}(x)=-\infty.$ Straightforward calculations show that $$h_{\sigma}'(x)=\frac{-4x^2+(4-R_0^Pc_0)Nx+(R_0^P-1)c_0N^2}{(N-x)^2x}.$$ Next, we prove this theorem case by case.
  \begin{enumerate}
  \item[(1).]
  For $R_0^P<1,$ we have $$\lim_{x\downarrow0}h_{\sigma}(x)=\infty,$$ which implies $$\lim_{x\downarrow0}p^s_{\sigma}(x)=\infty.$$By $R_0^S=R_0^P+c_0^{-1}>1,$ we have $$\frac{-4(\sqrt{c_0}+1)}{c_0}<1-\frac{1}{c_0}<R_0^P.$$ Hence if $R_0^P<1$ and $\displaystyle R_0^P\leqslant\displaystyle\frac{4(\sqrt{c_0}-1)}{c_0},$ we have $$(4-R_0^Pc_0)^2+16(R_0^P-1)c_0\leqslant0.$$ Thus $$h_{\sigma}'(x)\leqslant0\ \text{in}\ (0,N),$$ which indicates that $p_{\sigma}^s$ is strictly decreasing in $(0,N).$

  If $R_0^P<1$ and $R_0^P>\displaystyle\frac{4(\sqrt{c_0}-1)}{c_0},$ then $$h_{\sigma}'(x)<0\ \text{in}\ (0,I_-),\ (I_+,N),\ \text{and}\ h_{\sigma}'(x)>0\ \text{in}\ (I_-,I_+),$$  where $I_{\pm}$ is defined in \eqref{+}. Hence $p_{\sigma}^s$ is strictly decreasing in $(0,I_-),$ $(I_+,N),$ and increasing in $(I_-,I_+).$
  \item[(2).]
  For $R_0^P=1,$ $$\lim_{x\to0^+}h_{\sigma}(x)=-4\ln N,$$ which implies $$\lim_{x\downarrow0}p_{\sigma}^s(x)=C/N.$$ If $R_0^P=1$ and  $c_0\geqslant4,$  then $$h_{\sigma}'(x)<0\ \text{in}\ (0,N),$$ which indicates that $p_{\sigma}^s$ is strictly decreasing in  $(0,N).$

  If $R_0^P=1$ and  $c_0<4,$  then $$h_{\sigma}'(x)>0\ \text{in}\ \Bigg(0,\bigg(1-\frac{c_0}{4}\bigg)N\Bigg),\ \text{and}\ h_{\sigma}'(x)<0\ \text{in}\ \Bigg(\bigg(1-\frac{c_0}{4}\bigg)N,N\Bigg).$$ Hence $p_{\sigma}^s$ is strictly increasing in $\Bigg(0,\bigg(1-\displaystyle\frac{c_0}{4}\bigg)N\Bigg)$ and strictly decreasing in $\Bigg(\bigg(1-\displaystyle\frac{c_0}{4}\bigg)N,N\Bigg).$
  \item[(3).]
  For $R_0^P>1,$  $$h_{\sigma}'(x)>0\ \text{in}\ (0,I_*(\sigma)),\  \text{and}\ h_{\sigma}'(x)<0\ \text{in}\ (I_*(\sigma),N),$$ where $I_*(\sigma)$ is defined in \eqref{-}. Hence
  $p^s_{\sigma}$ is strictly increasing in $(0,I_*({\sigma}))$ while decreasing in $(I_*({\sigma}),N),$ and  peaks at $I_*(\sigma).$ Moreover, $\lim_{x\downarrow0}h_{\sigma}(x)=-\infty$ implies $$\lim_{x\downarrow0}p_{\sigma}^s(x)=0.$$
\end{enumerate}

Now we complete the proof.
\end{proof}
Figure~\ref{fig1} well illustrates results in Theorem~\ref{th5}.

Although the mean and variance of the stationary distribution are given in \cite{G}, for sake of integrity, we restate them as a corollary of Theorem~\ref{th5} and give an alternative proof by considering the FPE \eqref{10} directly.

\vskip 0.2cm
\begin{corollary}[\textbf{Mean and Variance}]\hfill
Assume that $R_0^S>1.$  Then the mean and variance of the invariant density $p_{\sigma}^s$ are given by
  \begin{equation}
    \label{6}
    {\rm E}[I]=I^*(\sigma)
  \end{equation} and
  \begin{equation}
    \label{24}
    {\rm Var}[I]=(I^*-I^*(\sigma))I^*(\sigma),
  \end{equation} where $I^*(\sigma)=\Bigg(1-\displaystyle \frac{1}{R_0^D+1-\frac{R_0^D}{R_0^S}}\Bigg)N$ and $I^*=\bigg(1-\displaystyle \frac{1}{R_0^D}\bigg)N.$
\end{corollary}
\newpage
\begin{figure}
\centering
    \subfigure[$R_0^P<1$ and $R_0^P\leqslant\frac{4(\sqrt{c_0}-1)}{c_0}.$]{\includegraphics[width=2in,height=1.2in]{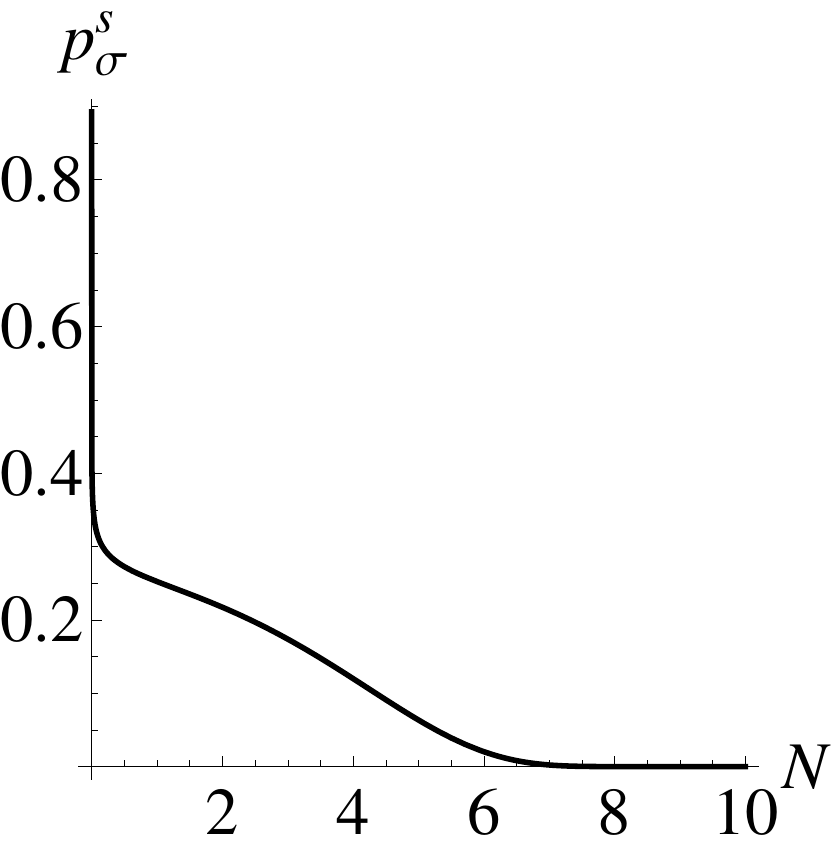}} \hspace{1cm}
    \subfigure[$\frac{4(\sqrt{c_0}-1)}{c_0}<R_0^P<1.$]{\includegraphics[width=2in,height=1.2in]{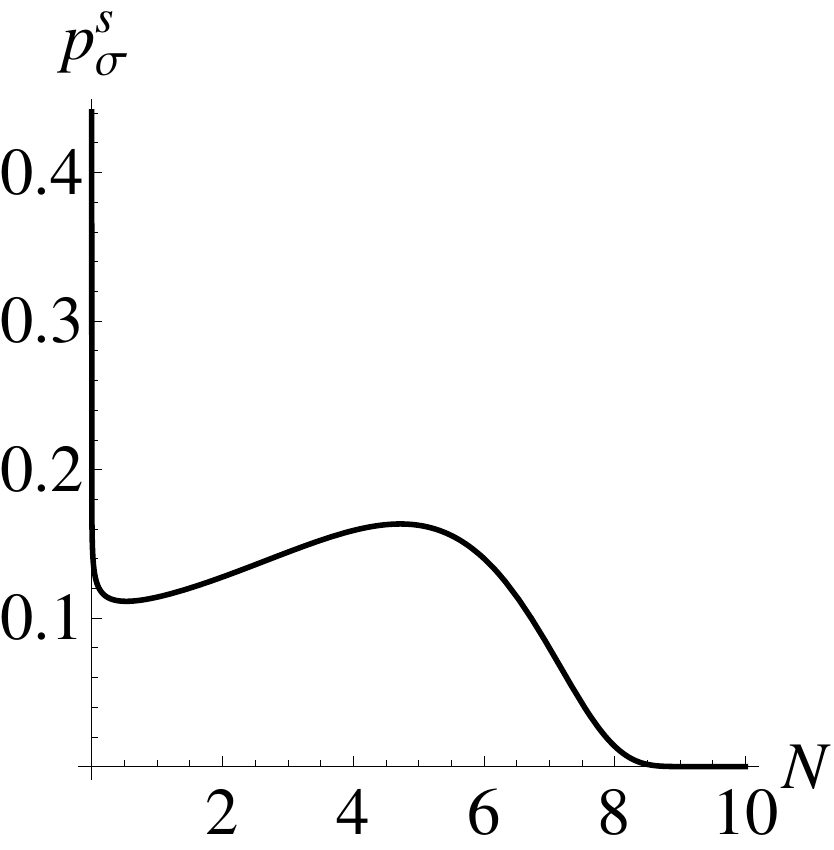}} \\
    \subfigure[$R_0^P=1\leqslant\frac{c_0}{4}.$]{\includegraphics[width=2in,height=1.2in]{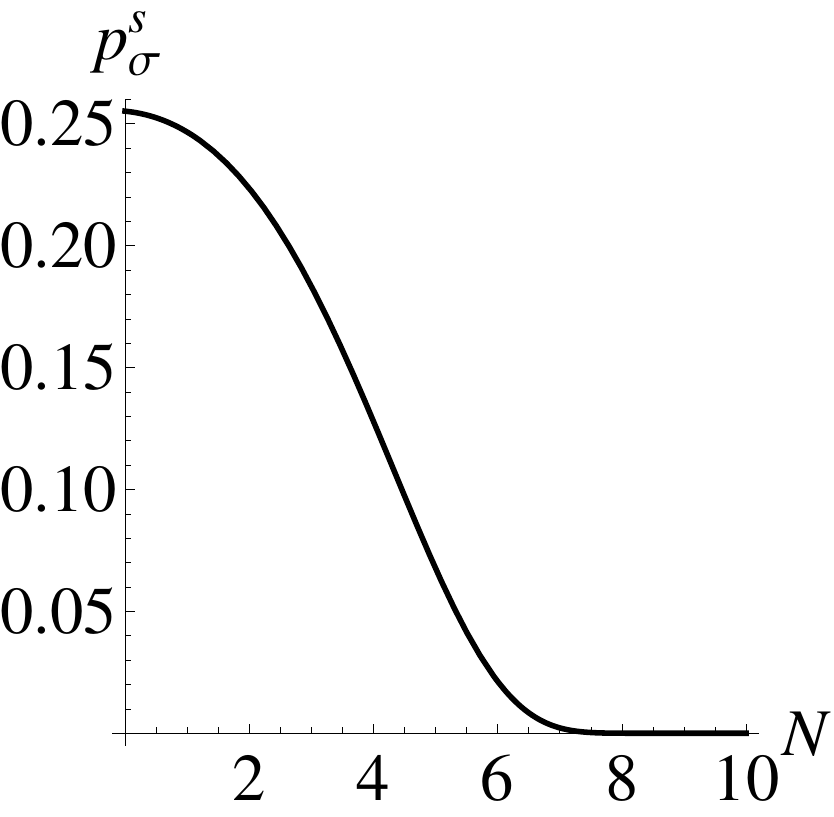}} \hspace{1cm}
    \subfigure[$R_0^P=1>\frac{c_0}{4}.$ ]{\includegraphics[width=2in,height=1.2in]{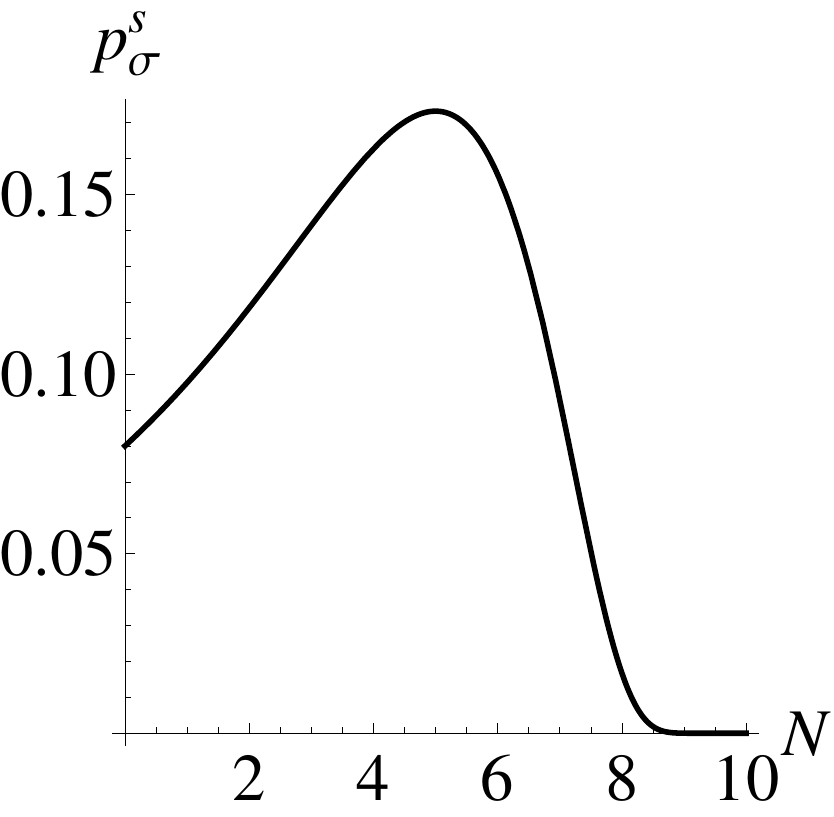}} \\
    \subfigure[$R_0^P>1.$]{\includegraphics[width=2in,height=1.2in]{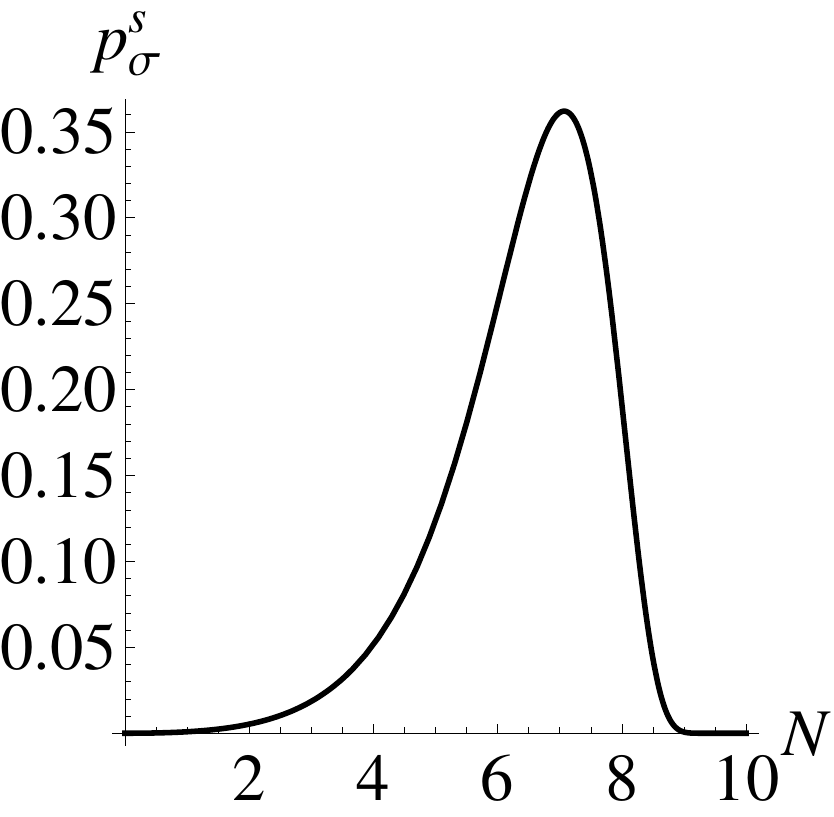}}
\caption{\small Profiles of invariant density $p^s_{\sigma}$ in different cases.} \label{fig1}
\end{figure}

\vskip 0.2cm
\begin{proof}
  It follows from \eqref{10}  that
\begin{equation}
  \label{25}
  x[\beta N-\mu-\gamma-\beta x]p_{\sigma}^s(x)-\frac{1}{2}\sigma^2\frac{{\rm d}}{{\rm d}x}(x^2(N-x)^2p_{\sigma}^s(x))=0.
\end{equation} Note from \eqref{17} and Theorem~\ref{th5}, \begin{equation}
  \label{q}\lim_{x\uparrow N}p_{\sigma}^s(x)=\lim_{x\downarrow0}xp_{\sigma}^s(x)=0.
\end{equation}
Integrating on both sides of \eqref{25} from $0$ to $N,$ we have \begin{equation*}
  \int_0^Nx[\beta N-\mu-\gamma-\beta x]p_{\sigma}^s(x){\rm d}x=0,
\end{equation*}
which implies
\begin{equation}
  \label{26}
  (\beta N-\mu-\gamma)\int_0^Nxp_{\sigma}^s(x){\rm d}x=\beta\int_0^Nx^2p_{\sigma}^s(x){\rm d}x.
\end{equation}
Multiplying $\displaystyle\frac{1}{x}$ on both sides of \eqref{25} and integrating from $0$ to $N,$ we have
\begin{equation}\label{27}
  \int_0^N[\beta N-\mu-\gamma-\beta x]p_{\sigma}^s{\rm d}x-\frac{1}{2}\sigma^2\bigg[x(N-x)^2p_{\sigma}^s(x)\Big|_0^N+\int_0^N(N-x)^2p_{\sigma}^s(x){\rm d}x\bigg]=0.
\end{equation} Combining \eqref{q}-\eqref{27}, we have
$${\rm E}[I]=\int_0^Nxp_{\sigma}^s{\rm d}x=I^*(\sigma)$$ and $${\rm Var}[I]=\int_0^Nx^2p_{\sigma}^s{\rm d}x-\Big(\int_0^Nxp_{\sigma}^s{\rm d}x\Big)^2=(I^*-I^*(\sigma))I^*(\sigma).$$
\end{proof}
From Theorem~\ref{th5}, it follows the persistence threshold theorem with respect to disease-extinction with large probability and disease-persistence with large probability.
\vskip 0.2cm
\begin{theorem} [\textbf{Persistence Threshold Theorem}]\label{th6}
   Assume $R_0^S>1.$
  \begin{enumerate}
    \item[{\rm(i)}] If $R_0^P<1$ or $R_0^P=1\leqslant\displaystyle\frac{c_0}{4},$ then \eqref{1} is \textit{disease-extinct} with large probability.
    \item[{\rm(ii)}] If $R_0^P>1$ or $R_0^P=1>\displaystyle\frac{c_0}{4},$ then \eqref{1} is \textit{disease-persistent} with large probability.
  \end{enumerate}
\end{theorem}
For $R_0^S>1,$ even there exists a stationary distribution as proved in \cite{G}, the chance for the disease to go extinct may still be very large (or the chance for the disease to persist may still be tiny) if $R_0^P<1$ or $R_0^P=1\leqslant\displaystyle\frac{c_0}{4}.$ The disease will persist with large probability  only when $R_0^P>1$ or $R_0^P=1>\displaystyle\frac{c_0}{4}.$ Such result interprets differently from deterministic epidemic models in that for $R_0^D>1,$ the disease is endemic and there is no chance for the disease to vanish. However, in the stochastic epidemic model, even when the stochastic basic reproduction number $R_0^S>1,$ there is still chance for the disease to go extinct. Model \eqref{0} is also different from stochastic epidemic models in term of a finite Markov chain, which concludes that disease goes extinct with probability one regardless of the basic reproduction number \cite{A}.

From Theorem~\ref{th6}, for $R_0^P>1,$ the mode of the stationary distribution is located at $I_*(\sigma).$ We define the \textit{stochastic disease prevalence} for the SDE SIS model \eqref{0} by $$\displaystyle\frac{I_*(\sigma)}{N}=\frac{1}{8}\bigg[4-R_0^Pc_0+\sqrt{(4+R_0^Pc_0)^2-16c_0}\bigg],$$ a number independent of the total population size $N.$ Recall that for the deterministic model, the disease prevalence is $\displaystyle\frac{I^*}{N}=1-\frac{1}{R_0^D},$ where $R_0^D=\displaystyle\frac{\beta  N}{\mu+\gamma}$ is the deterministic basic reproduction number and $I^*=N\bigg(1-\displaystyle\frac{1}{R_0^D}\bigg).$ In the following, we compare the two numbers for $R_0^P>1.$
\vskip 0.2cm
\begin{figure}
\centering
    \includegraphics[width=2in,height=1.2in]{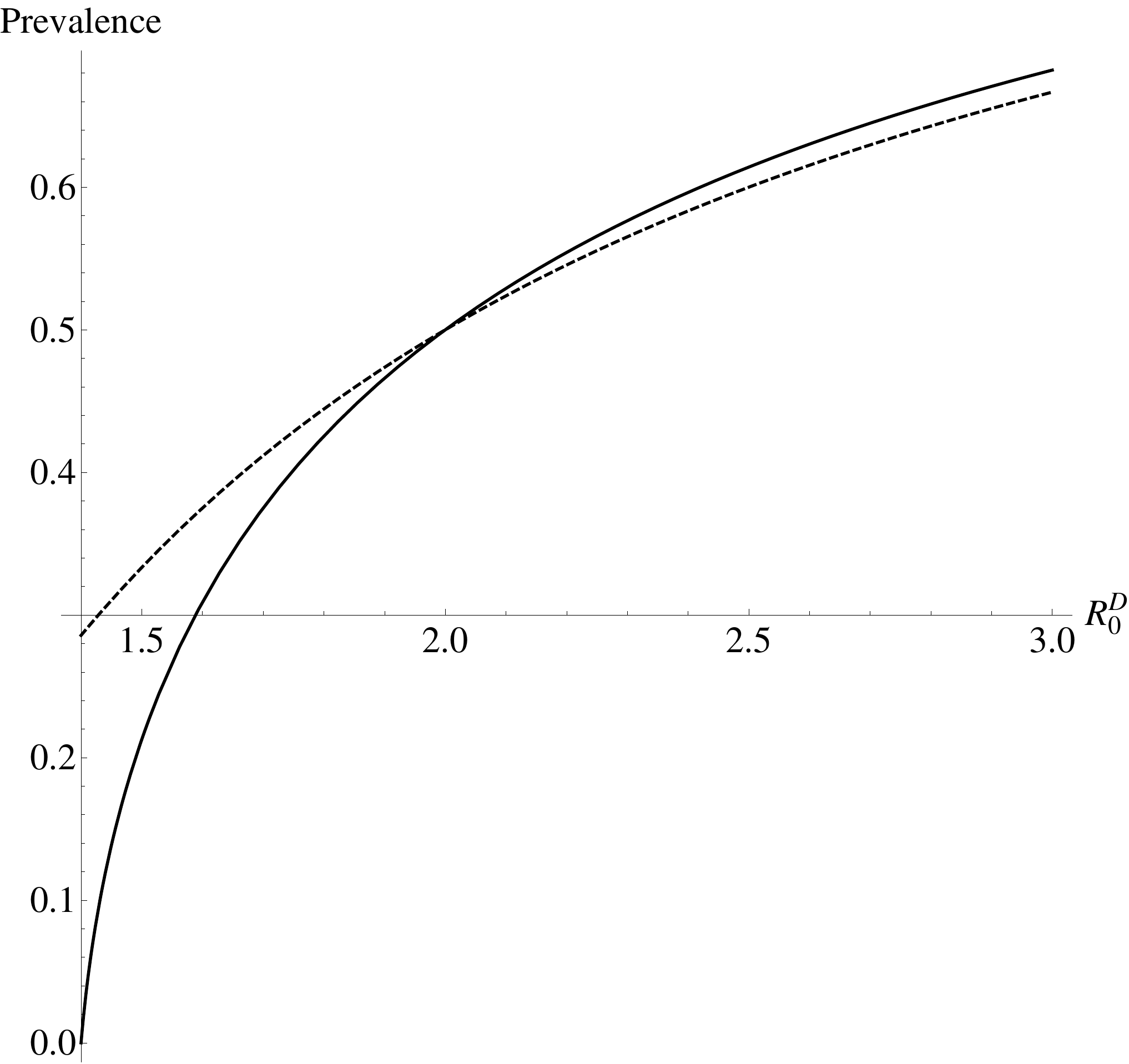}
\caption{\small Disease prevalence. The dashed line and the solid line stand for the deterministic disease prevalence curve  and the stochastic disease  prevalence curve, respectively. The two curves intersect at $R_0^D=2.$} \label{fig2}
\end{figure}
\begin{theorem}[\textbf{Disease Prevalence}]\label{th12} Assume $R_0^P>1.$
\begin{enumerate}
  \item[{\rm (1)}] If $R_0^D<2,$ then  $$\displaystyle\frac{I_*(\sigma)}{N}<\frac{I^*}{N},$$ i.e., the stochastic prevalence is smaller than the deterministic prevalence.
  \item[{\rm (2)}] If $R_0^D=2,$ then $$\displaystyle\frac{I_*(\sigma)}{N}=\frac{I^*}{N},$$ i.e., the stochastic prevalence equals the deterministic prevalence.
  \item[{\rm (3)}] If $R_0^D>2,$ then $$\displaystyle\frac{I_*(\sigma)}{N}>\frac{I^*}{N},$$ i.e., the stochastic prevalence is bigger than the deterministic prevalence.
\end{enumerate}
   \end{theorem}
   \vskip 0.2cm
\begin{proof}
  Equality $c_0R_0^D=c_0R_0^P+2$ implies $$\frac{I_*(\sigma)}{N}=\displaystyle\frac{1}{8}\bigg[6-c_0 R_0^D+\sqrt{(c_0R_0^D+2)^2-16c_0}\bigg].$$
  Notice that $\displaystyle\frac{I_*(\sigma)}{N}<\frac{I^*}{N}$ is equivalent to \begin{equation}
    \sqrt{(c_0R_0^D+2)^2-16c_0}<2-\frac{8}{R_0^D}+c_0R_0^D.\label{4}
  \end{equation}Since $R_0^P=R_0^D-2c_0^{-1}>1,$ we have\[
  \begin{split}
    &2-\frac{8}{R_0^D}+c_0R_0^D\\
    =&\frac{c_0(R_0^D)^2+2R_0^D-8}{R_0^D}\\
    >&\frac{c_0(R_0^D)^2+2R_0^D-4c_0(R_0^D-1)}{R_0^D}\\
    =&\frac{c_0(R_0^D-2)^2+2R_0^D}{R_0^D}>0.
  \end{split}
  \]Hence inequality \eqref{4} is equivalent to $$-16c_0<\displaystyle\frac{64}{(R_0^D)^2}-\frac{16}{R_0^D}(2+c_0R_0^D).$$
  Note that   $$-16c_0-\displaystyle\frac{64}{(R_0^D)^2}+\frac{16}{R_0^D}(2+c_0R_0^D)=\frac{32(R_0^D-2)}{(R_0^D)^2},$$ we have $$\frac{I_*(\sigma)}{N}<\frac{I^*}{N}\Leftrightarrow R_0^D<2.$$
\end{proof}

Figure~\ref{fig2} demonstrates results in Theorem~\ref{th12}.
From Theorem~\ref{th12}, we see that though the randomness in the transmission coefficient stabilizes the disease-free equilibrium  for $R_0^S<1$ and still reduces severity of disease for $1<R_0^P<R_0^D<2,$ it enhances severity for $R_0^D>2.$

In other words, when the deterministic basic reproduction number is small, stochasticity in the transmission rate decreases the prevalence. However, when the deterministic basic reproduction number is big, such stochasticity in fact increases the prevalence.

Persistence result in Theorem~\ref{th6} is different from that stated in \textit{Theorem~5.1} in \cite{G}, which claims that almost all the sample paths will fluctuate around the level $\tilde{I}_*(\sigma)$  in the sense of inequalities \eqref{p}. Whereas Theorem~\ref{th6} states that it is most likely that the disease persists at the level $I_*(\sigma).$ In other words, the most possible number of infected individuals in the total population at the stationary distribution is $I_*(\sigma).$

Now we compare $\tilde{I}_*(\sigma)$ and $I_*(\sigma).$
\vskip 0.2cm
\begin{theorem}\label{th4}
   Assume $R_0^P>1.$ Then\begin{enumerate}
     \item[{\rm (a)}] $\tilde{I}_*(\sigma)>I_*(\sigma)$ if $\displaystyle\frac{10}{3}c_0^{-1}<R_0^D<\frac{3}{2}+\frac{2}{3}c_0^{-1};$\\
     \item[{\rm (b)}] $\tilde{I}_*(\sigma)=I_*(\sigma)$ if $R_0^D=\displaystyle\frac{3}{2}+\frac{2}{3}c_0^{-1}<\frac{7}{4};$\\
     \item[{\rm (c)}] $\tilde{I}_*(\sigma)<I_*(\sigma)$ if $1+2c_0^{-1}<R_0^D\leqslant\displaystyle\frac{10}{3}c_0^{-1}.$
   \end{enumerate}
\end{theorem}
\vskip 0.2cm
\begin{proof}
   Recall that $R_0^P=R_0^D-2c_0^{-1}.$ Rewrite $\tilde{I}_*(\sigma)$ and $I_*(\sigma),$ we have $$\tilde{I}_*(\sigma)=\displaystyle\frac{N}{2}\bigg[\sqrt{(c_0R_0^D)^2-4c_0}-c_0R_0^D+2\bigg]$$ and $$I_*(\sigma)=\displaystyle\frac{N}{2}\Bigg[\sqrt{\bigg(\frac{2+c_0R_0^D}{4}\bigg)^2-c_0}+\frac{3}{2}-\frac{c_0R_0^D}{4}\Bigg].$$ Hence
   $$\tilde{I}_*(\sigma)>I_*(\sigma)
     \Leftrightarrow\sqrt{(c_0R_0^D)^2-4c_0}>\frac{3c_0R_0^D-2}{4}+\sqrt{\bigg(\frac{2+c_0R_0^D}{4}\bigg)^2-c_0}.$$
     Note that $$\frac{3c_0R_0^D-2}{4}=\frac{3c_0R_0^P+4}{4}>0.$$ This implies \[
   \begin{split}
   &\tilde{I}_*(\sigma)>I_*(\sigma)\\
  \Leftrightarrow&\frac{3}{8}(c_0R_0^D)^2+\frac{c_0R_0^D}{2}-\frac{1}{2}-3c_0>2\bigg(\frac{3}{4}c_0R_0^D-\frac{1}{2}\bigg)\sqrt{\bigg(\frac{2+c_0R_0^D}{4}\bigg)^2-c_0}\\
  \Leftrightarrow&3c_0<\frac{3c_0R_0^D-2}{2}\Bigg[\frac{2+c_0R_0^D}{4}-\sqrt{\bigg(\frac{2+c_0R_0^D}{4}\bigg)^2-c_0}\Bigg]\\
  \Leftrightarrow&3c_0<\frac{c_0(3c_0R_0^D-2)/2}{\displaystyle\frac{2+c_0R_0^D}{4}+\sqrt{\bigg(\frac{2+c_0R_0^D}{4}\bigg)^2-c_0}}\\
  \Leftrightarrow&\sqrt{\bigg(\frac{2+c_0R_0^D}{4}\bigg)^2-c_0}<\frac{3c_0R_0^D-10}{12}\\
  \Leftrightarrow&\frac{10}{3}c_0^{-1}<R_0^D<\frac{3}{2}+\frac{2}{3}c_0^{-1}.
   \end{split}
   \] Combining $R_0^P=R_0^D-2c_0^{-1}>1,$ we arrive at the conclusions.
\end{proof}
\begin{figure}
\centering
\subfigure[Disease prevalence surfaces.]{\includegraphics[width=2.2in,height=1.2in]{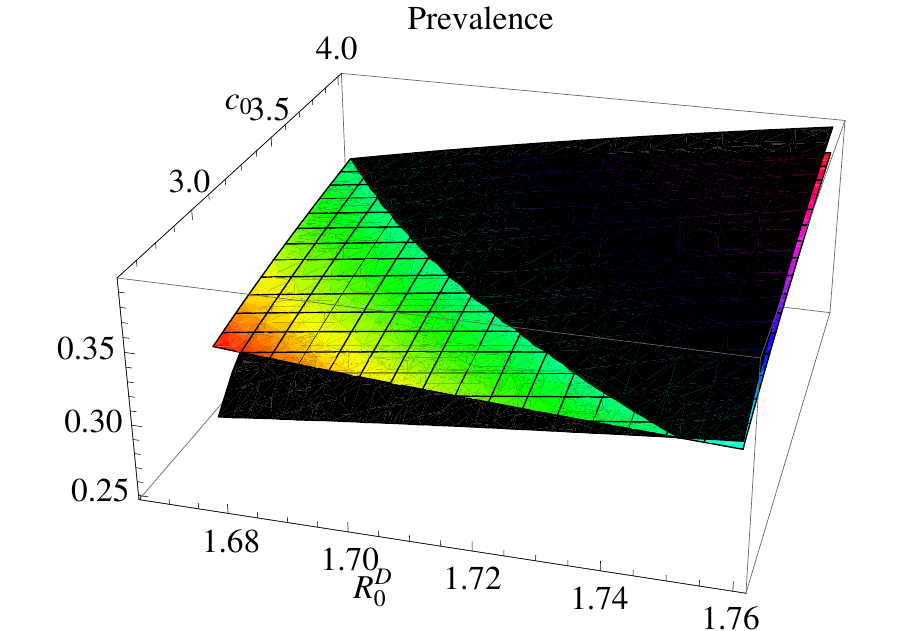}}  \hspace{1cm}
    \subfigure[$c_0=3.$]{\includegraphics[width=2in,height=1.2in]{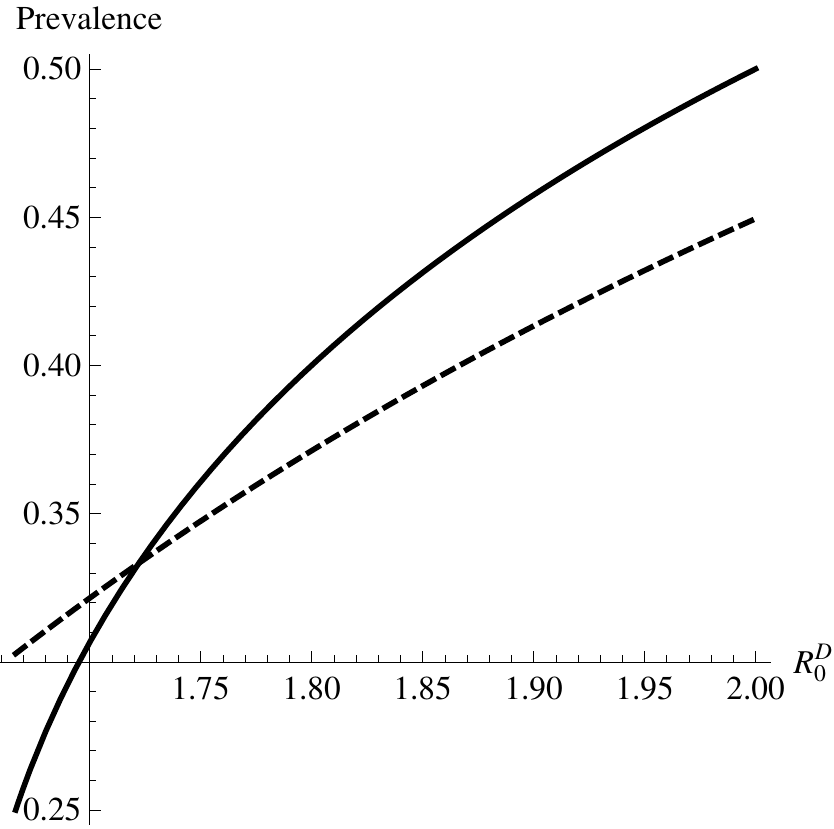}}\\
    \subfigure[$R_0^D=13/8.$]{\includegraphics[width=2in,height=1.2in]{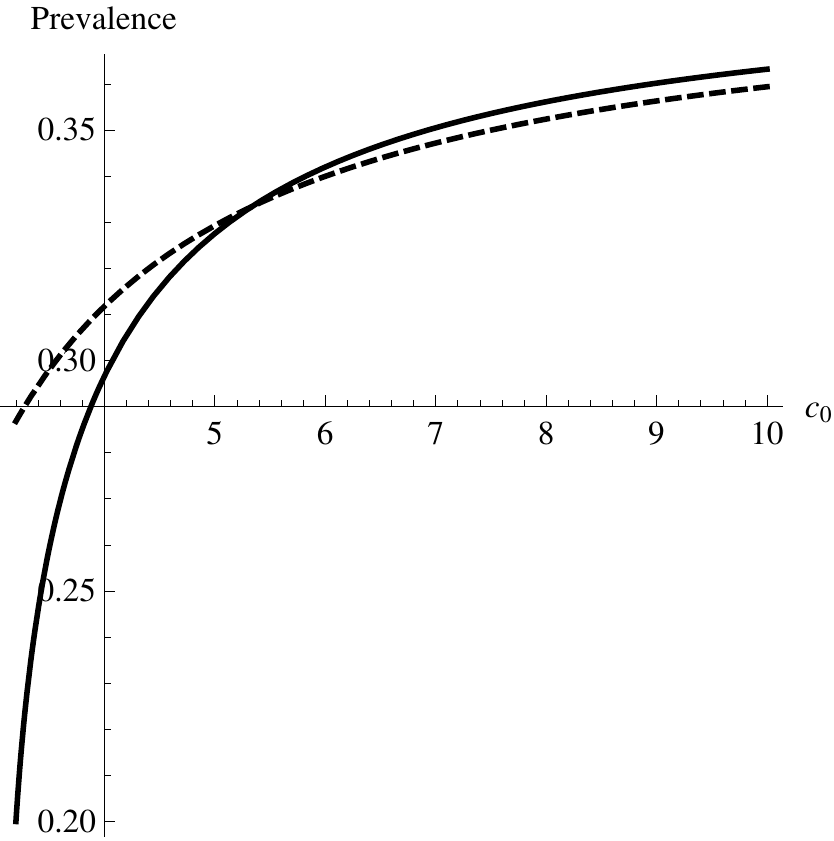}}\hspace{1cm}     \subfigure[$R_0^D=13/8.$]{\includegraphics[width=2in,height=1.2in]{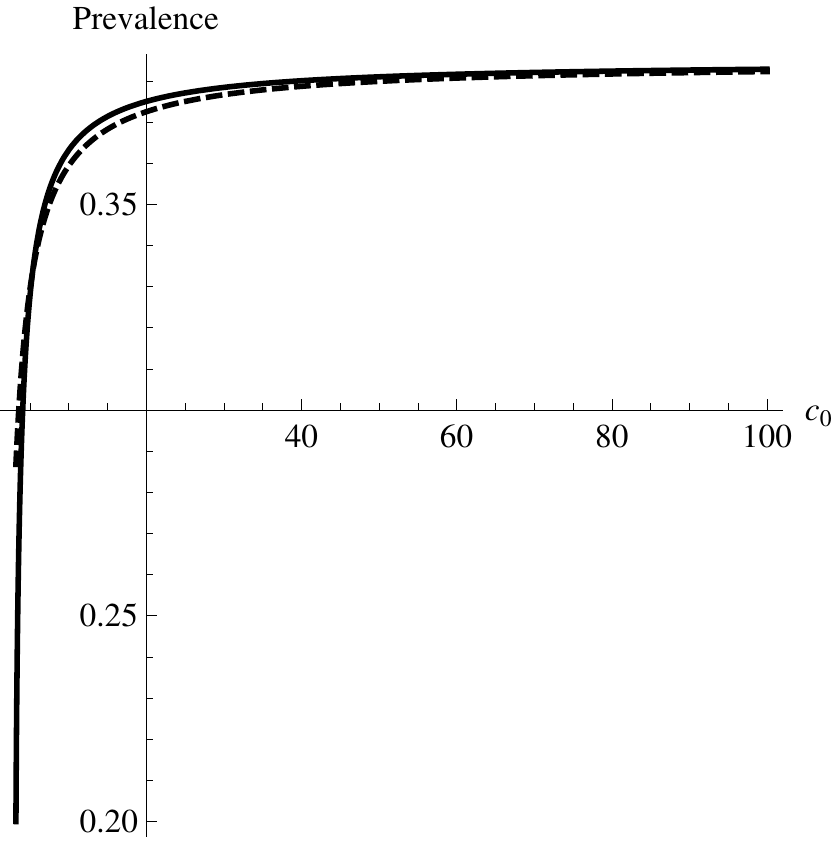}}
\caption{Disease prevalence. In {\rm (a)}, the surface in mixed color stands for $\tilde{I}_*(\sigma)/N$ and the surface in black stands for $I_*(\sigma)/N.$ In {\rm (b)}, {\rm (c)} and {\rm (d)}, the dashed line and the solid line stand for $\tilde{I}_*(\sigma)/N$ and $I_*(\sigma)/N,$ respectively. Two curves almost coincide when $c_0$ is sufficiently large.} \label{fig3}
\end{figure}
Results in Theorem~\ref{th4} are illustrated in Figure~\ref{fig3}. We see that when $c_0$ are sufficiently large (or equivalently, $\sigma$ is sufficiently small, ), $\tilde{I}_*(\sigma)$ and $I_*(\sigma)$ almost coincide. In fact, it is easy to verify that $$\lim_{\sigma\to0}\tilde{I}_*(\sigma)=\lim_{\sigma\to0}I_*(\sigma)=I^*,$$ which is a corollary of Theorem~\ref{th7} in the next section.


\section{Limit Stochastic Threshold Dynamics}\label{sect5}

In this section, we further investigate the asymptotic dynamics of \eqref{1} and try to establish a sharp connection between \eqref{1} and its deterministic counterpart \eqref{ode} in term of the limit of the invariant density $p_{\sigma}^s$ as $\sigma\to0.$
\vskip 0.2cm
\begin{theorem}[\textbf{Limit Stochastic Threshold Theorem}]\hfill
  \label{th7}
  \begin{enumerate}
    \item[{\rm(a).}] If $R_0^D\leqslant1,$ then the number of susceptibles $I$ at DEF converges to $0$ in probability as $\sigma\to0.$
    \item[{\rm(b).}] If $R_0^D>1,$ then the number of susceptibles $I$ at $p_{\sigma}^s$ converges to $I^*$ in probability as $\sigma\to0.$
  \end{enumerate}
\end{theorem}
\vskip 0.2cm
\begin{proof} Note for $R_0^D\leqslant1,$ $$R_0^S<1,\ \forall\ \sigma>0.$$ Conclusion (a) trivially follows from Theorem~\ref{th1} for the random variable $I$ at DFE is identically zero.

In the following, we only prove (b).

It suffices to show that for small $\varepsilon>0,$
  \begin{equation}\label{14}
    \int_{I^*-\varepsilon}^{I^*+\varepsilon}p_{\sigma}^s(x){\rm d}x\to1,\ \text{as}\ \sigma\to0.
  \end{equation}
Since \begin{equation}
  \label{31}
  \lim_{\sigma\to0}I_*(\sigma)=I^*,
\end{equation}for sufficiently small $\sigma,$ \begin{equation}
  \label{13}
  |I_*(\sigma)-I^*|<\frac{\varepsilon}{2}.
\end{equation}
This implies \begin{equation}
  \label{30}
  I^*-\varepsilon<I_*(\sigma)-\frac{\varepsilon}{2},\ I^*+\varepsilon>I_*(\sigma)+\frac{\varepsilon}{2}.
\end{equation} Since $R_0^D>1$ and $R_0^P=R_0^D-\displaystyle\frac{\sigma^2N^2}{\mu+\gamma},$ we know that $R_0^P>1$ for all sufficiently small $\sigma^2.$ By Theorem~\ref{th5}, $p_{\sigma}^s$ is increasing in $(0,I_*(\sigma))$ and decreasing in $(I_*(\sigma),N).$ Thus
\[
\begin{split}
 &\int_0^{I_*(\sigma)-\varepsilon/2}p_{\sigma}^s(x){\rm d}x\\<&(I_*(\sigma)-\varepsilon/2)p_{\sigma}^s(I_*(\sigma)-\varepsilon/2)\\
 =&(I_*(\sigma)-\varepsilon/2)p_{\sigma}^s(I_*(\sigma)-\varepsilon/4)e^{h_{\sigma}(I_*(\sigma)-\varepsilon/2)-h_{\sigma}(I_*(\sigma)-\varepsilon/4)}\\
 <&\frac{I_*(\sigma)-\varepsilon/2}{\varepsilon/4}e^{h_{\sigma}(I_*(\sigma)-\varepsilon/2)-h_{\sigma}(I_*(\sigma)-\varepsilon/4)}\int_{I_*(\sigma)-\varepsilon/4}^{I_*(\sigma)}p_{\sigma}^s(x){\rm d}x\\
 <&\frac{I_*(\sigma)-\varepsilon/2}{\varepsilon/4}e^{h_{\sigma}(I_*(\sigma)-\varepsilon/2)-h_{\sigma}(I_*(\sigma)-\varepsilon/4)}\int_{I_*(\sigma)-\varepsilon/2}^{I_*(\sigma)+\varepsilon/2}p_{\sigma}^s(x){\rm d}x.
\end{split}
\]
Similarly, $$\int_{I_*(\sigma)+\varepsilon/2}^Np_{\sigma}^s(x){\rm d}x<\frac{N-I_*(\sigma)-\varepsilon/2}{\varepsilon/4}e^{h_{\sigma}(I_*(\sigma)+\varepsilon/2)-h_{\sigma}(I_*(\sigma)+\varepsilon/4)}\int_{I_*(\sigma)-\varepsilon/2}^{I_*(\sigma)+\varepsilon/2}p_{\sigma}^s(x){\rm d}x.$$
Hence by \eqref{30}, we see\[
\begin{split}
  1=&\int_0^{I_*(\sigma)-\varepsilon/2}p_{\sigma}^s(x){\rm d}x+\int_{I_*(\sigma)+\varepsilon/2}^Np_{\sigma}^s(x){\rm d}x+\int_{I_*(\sigma)-\varepsilon/2}^{I_*(\sigma)+\varepsilon/2}p_{\sigma}^s(x){\rm d}x\\
 <&\Bigg[1+\frac{I_*(\sigma)-\varepsilon/2}{\varepsilon/4}e^{h_{\sigma}(I_*(\sigma)-\varepsilon/2)-h_{\sigma}(I_*(\sigma)-\varepsilon/4)}\\
 &+\frac{N-I_*(\sigma)-\varepsilon/2}{\varepsilon/4}e^{h_{\sigma}(I_*(\sigma)+\varepsilon/2)-h_{\sigma}(I_*(\sigma)+\varepsilon/4)}\Bigg]
 \cdot\int_{I_*(\sigma)-\varepsilon/2}^{I_*(\sigma)+\varepsilon/2}p_{\sigma}^s(x){\rm d}x\\
  <&\Bigg[1+\frac{I_*(\sigma)-\varepsilon/2}{\varepsilon/4}e^{h_{\sigma}(I_*(\sigma)-\varepsilon/2)-h_{\sigma}(I_*(\sigma)-\varepsilon/4)}\\
  &+\frac{N-I_*(\sigma)-\varepsilon/2}{\varepsilon/4}e^{h_{\sigma}(I_*(\sigma)+\varepsilon/2)-h_{\sigma}(I_*(\sigma)+\varepsilon/4)}\Bigg]\cdot\int_{I^*-\varepsilon}^{I^*+\varepsilon}p_{\sigma}^s(x){\rm d}x.
\end{split}
\]
In the following, we only need to show that \[\begin{split}&\frac{I_*(\sigma)-\varepsilon/2}{\varepsilon/4}e^{h_{\sigma}(I_*(\sigma)-\varepsilon/2)-h_{\sigma}(I_*(\sigma)-\varepsilon/4)}\\+&\frac{N-I_*(\sigma)-\varepsilon/2}{\varepsilon/4}e^{h_{\sigma}(I_*(\sigma)+\varepsilon/2)-h_{\sigma}(I_*(\sigma)+\varepsilon/4)}\to0,\ \text{as}\ \sigma\to0.\end{split}\] By \eqref{31}, it suffices to prove that
\begin{equation}
  \label{32}
  h_{\sigma}(I_*(\sigma)-\varepsilon/2)-h_{\sigma}(I_*(\sigma)-\varepsilon/4)\to-\infty
\end{equation}
and
\begin{equation}
  \label{33}
  h_{\sigma}(I_*(\sigma)+\varepsilon/2)-h_{\sigma}(I_*(\sigma)+\varepsilon/4)\to-\infty.
\end{equation}
Straightforward calculations give
$$h_{\sigma}''(I_*(\sigma))=-\frac{N}{I_*(\sigma)(N-I_*(\sigma))^2}\sqrt{(4+R_0^Pc_0)^2-16c_0}.$$ Since\begin{equation}
  \label{34}
c_0\to\infty,\ \text{as}\ \sigma\to0,
\end{equation}which implies
$$\lim_{\sigma\to0}h_{\sigma}''(I_*(\sigma))=-\infty.$$
Hence for small $\varepsilon>0$ and sufficiently small $\sigma>0,$ $h'_{\sigma}$ is decreasing in $(I_*(\sigma)-\varepsilon/2,I_*(\sigma)+\varepsilon/2).$
Thus $$h_{\sigma}(I_*(\sigma)-\varepsilon/2)-h_{\sigma}(I_*(\sigma)-\varepsilon/4)<-\frac{\varepsilon}{4}h'_{\sigma}(I_*(\sigma)-\varepsilon/4).$$ Recall that $$h_{\sigma}'(x)=\frac{-4x^2+(4-R_0^Pc_0)Nx+(R_0^P-1)c_0N^2}{(N-x)^2x}.$$ Substituting $h_{\sigma}'(I_*(\sigma))=0,$ by \eqref{34}, we have
\[
h'_{\sigma}(I_*(\sigma)-\frac{\varepsilon}{4})=\frac{-\varepsilon^2+8\varepsilon I_*(\sigma)+\varepsilon N(R_0^Pc_0-4)}{4[N-(I_*(\sigma)-\varepsilon/4)]^2(I_*(\sigma)-\varepsilon/4)}\to\infty,\ \text{as}\ \sigma\to0.
\]
Now \eqref{32} is proved. Using similar arguments, we can prove \eqref{33}.
\end{proof}

Recall that for deterministic SIS model \eqref{ode}, we have the threshold theorem: the disease-free equilibrium $P_0=(N,0)$ is globally asymptotically stable if $R_0^D\leqslant1,$ and a unique endemic equilibrium $P_*=(N-I^*,I^*)$ is globally asymptotically stable if $R_0^D>1.$ Theorem~\ref{th7} establishes a sharp link between the SDE SIS model and ODE SIS model in term of their threshold dynamics as the randomness in the transmission coefficient vanishes.
\section{Summary}
In this paper, we further study the global dynamics of an SDE SIS epidemic model proposed in \cite{G}. Using Feller's test for explosions of solutions to one dimensional SDE, we establish a stochastic threshold theorem and thus prove the conjecture proposed in \cite{G}. By studying the FPE associated with the SDE, we prove the existence, uniqueness and global asymptotic stability of the invariant density of the FPE. Using the explicit formula for the invariant density, we define the persistence basic reproduction number and give the persistence threshold theorem in term of the invariant density. By comparing the stochastic disease prevalence with the deterministic disease prevalence, we discover that the stochastic  prevalence is bigger than the deterministic  prevalence if the deterministic basic reproduction number $R_0^D>2.$  This shows that the disease is more severe than predicted by the deterministic model if there is randomness in the transmission coefficient when the basic reproduction number is large enough. Finally we investigate the asymptotic dynamics of the SDE SIS model and establish a connection with the dynamics of the deterministic SIS model as the noise vanishes. We expect such global stochastic threshold theorem, persistence threshold theorem and limit stochastic threshold theorem discovered in this simple SIS type model to exist in more complicated epidemic models, for instance the SDE SIR model \cite{LJ,TBV}. We also expect that the approach used in this paper applies to other (high dimensional) SDE biological models, by considering the associated FPE and its invariant density, even though the dynamics may be more complex and there may not exist an explicit formula like \eqref{17} for the invariant density. Nevertheless, Feller's test may not work for  Stochastic Threshold Theorem (like Theorem~\ref{th1}) if the model can not be reduced to a one-dimensional SDE.
\section*{Acknowledgments} This research is supported by the financial support of a China Scholarship Council scholarship, University of Alberta Doctoral Recruitment Scholarship, Pundit RD Sharma Memorial Graduate Award, Eoin L. Whitney Scholarship  and Pacific Institute for the Mathematical Sciences(PIMS) Graduate Scholarship. The author is grateful to Professor Michael Yi Li at the University of Alberta and Professor Junjie Wei at Harbin Institute of Technology who introduced the author to mathematical epidemiology. The author is also indebt to two anonymous referees for their careful proofreading and valuable comments and suggestions which greatly enhances the presentation of this paper.


\section*{Appendix}\hfill


First, we present the monotonicity of function $g$ defined in \eqref{3} in Section 2.

\vskip 0.2cm
\begin{lemma}\label{le2}
The function $g$ defined in \eqref{3} is strictly increasing in $(0,N).$ Moreover, $$\lim_{\xi\downarrow0}g(\xi)=-\infty,\ \mbox{and}\ \lim_{\xi\uparrow N}g(\xi)=\infty.$$ Hence, its inverse function $g^{-1}$ exists and is strictly increasing in $(-\infty,\infty).$ In addition,  $$\lim_{y\to-\infty}g^{-1}(y)=0,\ \mbox{and}\ \lim_{y\to\infty}g^{-1}(y)=N.$$
\end{lemma}

Now we give some preliminary results on Feller's test of explosions and FPEs \cite{I,LM}.

The following is a result on Feller's test of explosions \cite{F,KS}. For a complete statement of Feller's test of explosions, we refer the reader to Proposition 5.22 on p. 345 in \cite{KS}.

Consider the following one-dimensional SDE
\begin{equation}\label{19}
\begin{split}
  {\rm d}X(t)=&b(X(t)){\rm d}t+a(X(t)){\rm d}B(t),\\
  X(0)=&X_0
\end{split}\end{equation}
for some $X_0\in\mathbb{R}.$ Assume that
\vskip 0.2cm
\noindent
{\rm(H1)}\ $(a(x))^2>0,\ \forall\ x\in\mathbb{R}.$
\vskip 0.2cm
\noindent
{\rm(H2)}\ $\forall\ x\in\mathbb{R},$ $\exists\ \varepsilon>0$ such that $\displaystyle\int_{x-\varepsilon}^{x+\varepsilon}\frac{1+|b(\xi)|}{\big(a(\xi)\big)^2}{\rm d}\xi<\infty.$

Define the \textit{scale function} $\psi$ by $$\psi(x)=\displaystyle\int_0^x\exp\Bigg(-\int_0^{\xi}\frac{2b(r)}{\big(a(r)\big)^2}{\rm d}r\Bigg){\rm d}\xi,\ x\in\mathbb{R}.$$ Now we give Feller's test \cite{KS}.
\vskip 0.2cm
\begin{lemma}\label{le3}
Assume that {\rm (H1)} and {\rm (H2)} hold. Let $X_0\in\mathbb{R}.$
\begin{enumerate}
\item[{\rm(1)}] If $\psi(-\infty)>-\infty$ and $\psi(\infty)=\infty,$ then $$\mathbb{P}\Big\{\lim_{t\to\infty}X(t)=-\infty\Big\}=1.$$
  \item[{\rm(2)}] If $\psi(-\infty)=-\infty$ and $\psi(\infty)=\infty,$ then $$\mathbb{P}\Big\{\underset{0\leqslant t<\infty}{\sup}X(t)=\infty\Big\}=\mathbb{P}\Big\{\underset{0\leqslant t<\infty}{\inf}X(t)=-\infty\Big\}=1.$$ In particular, the process $X_t$ is recurrent: for every $\xi\in\mathbb{R},$ $$\mathbb{P}\{X(t)=\xi:\ \exists\ t\in[0,\infty)\}=1.$$
\end{enumerate}
\end{lemma}

Next, we present some basic results on FPE.
Consider the following initial value Cauchy problem:
\begin{equation}
 \label{15} \frac{\partial v}{\partial t}=-\frac{\partial}{\partial x}(b(x)v)+\frac{1}{2}\frac{\partial^2}{\partial x^2}\Big(\big(a(x)\big)^2v\Big),\ \ t>0,\ x\in\mathbb{R},\end{equation}\begin{equation}\label{115}
  v(0,x)=v_0(x),\ x\in\mathbb{R}.
\end{equation}
Eq.\eqref{15} is the {\em Fokker-Planck} associated with \eqref{19}. It is well-known that there is a unique {\em generalized solution} to \eqref{15} provided that there exists a unique classical fundamental solution to \eqref{15} ( for the definitions of fundamental solution and generalized solution, see, for instance, p.365 and p.368 in \cite{LM}).

In the following, we assume there exists a unique classical fundamental solution to \eqref{15} and thus
\begin{theorem} {\rm \cite{LM}}
  \label{th0}
  For every $v_0\in L^1_+(\mathbb{R}),$  there exists a unique generalized solution $v(t,x)$ to \eqref{15}.
\end{theorem}

Eq. \eqref{15} defines the Markov semigroup (also called stochastic semigroup in \cite{LM}) $\{\mathcal{V}(t)\}_{t\geqslant0}$ of operators on $L^1_+(\mathbb{R})$ by
\begin{equation}
  \mathcal{V}(t)v_0(x)=v(t,x),
\end{equation} and the set of all densities $L^1_+(\mathbb{R})$  is invariant under $\mathcal{V}(t)$ for each $t>0.$

For more preliminaries on FPEs, we refer the reader to \cite{LM}.

Now we state a known sufficient condition given on p.742 in \cite{MLL} for the existence, uniqueness and global asymptotic stability of an invariant density (see also Theorem~2 and Theorem~4 on p.154 in \cite{I} and Theorem~11.9.1 on p.372 in \cite{LM}).

We define a {\em Lyapunov function} $V:\mathbb{R}\to\mathbb{R}$ as a $C^2$ function with the
following properties \cite{MLL}:
\begin{enumerate}
  \item[(P1)] $V(x)\ge0, \forall\ x\in\mathbb{R};$
  \item[(P2)] $\lim_{|x|\to\infty}V(x)=\infty;$
  \item[(P3)] $V(x),\ |V'(x)|\leqslant\delta_1e^{\delta_2|x|}, \forall\ x\in\mathbb{R},$ for some positive constants $\delta_1,\ \delta_2.$
\end{enumerate}
\vskip 0.2cm
\begin{theorem}\label{th9}
   Suppose there exists a Lyapunov function $V$ and positive constants $c_1,\ c_2,\ \delta,\ \alpha_1$ and $\alpha_2$ such that
\begin{equation}
  \label{18}
  \begin{split}
    &-\alpha_1+c_1|x|^{\delta}\leqslant V(x),\\
     &b(x)V'(x)+\frac{1}{2}a^2(x)V''(x)\leqslant-c_2V(x)+\alpha_2.
  \end{split}
\end{equation}
Then there exists a unique invariant probability measure $\rho$ for \eqref{19} which has the density $f_0$ with respect to the Lebesgue measure. Moreover, the invariant density $f_0$ is globally asymptotically stable in the sense that
  \begin{equation}
    \lim_{t\to\infty}\int_{-\infty}^{\infty}|\mathcal{V}(t)f(x)-f_0(x)|{\rm d}x=0,\ \forall\ f\in L^1_+(\mathbb{R}).
  \end{equation}
In addition, the process $X_t$ has the ergodic properties, i.e., for any $\rho$-integrable function $F:$
\begin{equation}
  \mathbb{P}_x\Big(\lim_{t\to\infty}\frac{1}{t}\int_0^tF(X_{\tau}){\rm d}\tau=\int_{-\infty}^{\infty}F(y)\rho({\rm d}y)\Big)=1,
\end{equation}for all $X_0=x\in\mathbb{R}.$
\end{theorem}
\begin{remark}
  In \cite{LM}, {\em regularity} of coefficients are required in Theorem~11.9.1. However, as pointed out on p. 365 in \cite{LM}, this regularity is simply for the existence and uniqueness of solutions to \eqref{15} and \eqref{115}. In other words, Theorem~11.9.1 on p.372 in \cite{LM} without regularity condition still holds if the existence and uniqueness of solutions to \eqref{15} and \eqref{115} are provided.
\end{remark}
In the following, we state the relationship between the solution $u(t,\xi)$ to \eqref{10} and the solution $p(t,x)$ to \eqref{5}.
\vskip 0.2cm
\begin{theorem}
  $$p(t,x)=g'(x)u(t,g(x))=\frac{N}{x(N-x)}u\Big(t,\log\frac{x}{N-x}\Big),\ \mbox{for\ all}\ t>0,\ x\in(0,N)$$ provided that the initial densities satisfy
  $$p_0(x)=\frac{N}{x(N-x)}u_0\Big(\log\frac{x}{N-x}\Big).$$ In particular, the invariant densities satisfy
  $$p^s_{\sigma}(x)=\frac{N}{x(N-x)}u^s_{\sigma}\Big(\log\frac{x}{N-x}\Big).$$
\end{theorem}
\begin{proof}
  Let $F_1(t,x)=\int_0^xp(t,y){\rm d}y$ and $F_2(t,\xi)=\int_{-\infty}^\xi u(t,\eta){\rm d}\eta.$ Then $$F_1(t,x)=F_2(t,g(x)),\ \forall\ t>0,\ x\in(0,N).$$ Hence for all $t>0,\ x\in(0,N),$ $$\frac{\partial F_1}{\partial x}(t,x)=\frac{\partial F_2}{\partial\xi}(t,g(x))g'(x),$$ i.e.,   $$p_{\sigma}(t,x)=g'(x)u(t,g(x)).$$
\end{proof}

\vskip 0.2cm
\begin{theorem}\label{th8}
\noindent\begin{enumerate}
\item[{\rm(1)}]
  $\kappa$ is an invariant measure for \eqref{2} if and only if $\nu=\kappa\circ g$ is an invariant measure for \eqref{1}.
  \item[{\rm(2)}] $u^s_{\sigma}$ is asymptotically stable if and only if $p^s_{\sigma}$ is asymptotically stable.
  \item[{\rm(3)}] $Y_t$ is ergodic if and only if $I_t$ is ergodic.
  \end{enumerate}
\end{theorem}
\vskip 0.2cm
\begin{proof}
We only prove case (2).
Suppose $u^s_{\sigma}$ is asymptotically stable, i.e.,  $$\lim_{t\to\infty}\int_{-\infty}^{\infty}|\mathcal{U}(t)f_0(\xi)-u^s_{\sigma}(\xi)|{\rm d}\xi=0,\ \forall\ f_0\in L^1_+(\mathbb{R}).$$ $\forall\ w_0\in L^1_+((0,N)),$ let $f_0(\xi)=[g^{-1}(\xi)]'w_0(g^{-1}(\xi)),$ then it is easy to verify that $f_0\in L^1_+(\mathbb{R}).$ Moreover, 
we can show that\[
    \int_0^N|\mathcal{P}(t)w_0(x)-p^s_{\sigma}(x)|{\rm d}x
    =\int_{-\infty}^{\infty}|\mathcal{U}(t)f_0(\xi)-u^s_{\sigma}(\xi)|{\rm d}\xi,
  \]which implies that $$\lim_{t\to\infty}\int_0^N|\mathcal{P}(t)w_0(x)-p^s_{\sigma}(x)|{\rm d}x=0,$$ i.e., $p^s_{\sigma}$ is asymptotically stable. Conversely, we can also prove that $u^s_{\sigma}$ is asymptotically stable provided that $p^s_{\sigma}$ is asymptotically stable.
\end{proof}



\begin{thebibliography}{99}
\bibitem{A} \newblock  L. J. S. Allen,
\newblock \emph{An Introduction to Stochastic Processes with Applications to Biology},
\newblock Pearson Prentice Hall, Upper Saddle River, 2003.

\bibitem{AM}\newblock  R. M. Anderson and R. M. May,
\newblock  \emph{Infectious Diseases of Humans: Dynamics and Control},
\newblock Oxford University Press, Oxford, 1992.

\bibitem{C} \newblock  G. Chen, T. Li, and C. Liu,
\newblock \emph{Lyapunov exponent of a stochastic SIRS model},
\newblock C. R. Acad. Sci. Paris, Ser. I, 351 (2013), 33--35.

\bibitem{DGM} \newblock  N. Dalal, D. Greenhalgh, and X. Mao,
\newblock \emph{A stochastic model of AIDS and condom use},
\newblock J. Math. Anal. Appl., 325 (2007), 36--53.

\bibitem{DH} \newblock  O. Diekmann and J. A. P. Heesterbeek,
\newblock \emph{Mathematical Epidemiology of Infectious Diseases: Model Building, Analysis
and Interpretation},
\newblock New York: Wiley, 2000.

\bibitem{DHM} \newblock  O. Diekmann, J. A. P. Heesterbeek, and J. A. J. Metz,
\newblock \emph{On the definition and the computation of the basic reproduction
ratio $R_0$ in models for infectious diseases in heterogeneous populations},
\newblock  J. Math. Biol., 28 (1990), 365--382.

\bibitem{F} \newblock  W. Feller,
\newblock \emph{Diffusion processes in one dimension},
\newblock Trans. Amer. Math. Soc., 77 (1954), 1--31.

\bibitem{G} \newblock  A. Gray et. al.,
\newblock \emph{A stochastic differential equation SIS epidemic model},
\newblock SIAM J. Appl. Math., 71 (2011), 876--902.
\bibitem{HY} \newblock  H. W. Hethcote and J. A. Yorke,
\newblock \emph{Gonorrhea Transmission Dynamics and Control},
\newblock Lecture Notes in Biomath. Vol. 56, Springer-Verlag, Berlin, 1984.

\bibitem{I} \newblock  H. M. I\^{t}o,
\newblock \emph{Ergodicity of randomly perturbed Lorenz model},
\newblock  J. Stat. Phys., 35 (1984), 151-158.


\bibitem{KS} \newblock I. Karatzas and S. Shreve,
\newblock \emph{Brownian Motion and Stochastic Calculus},
\newblock  Graduate Texts in Math. Vol. 113, 2nd ed., Springer Verlag, New York, 1997.

\bibitem{KM} \newblock W. O. Kermack and A. G. McKendrick,
\newblock \emph{A contribution to the mathematical theory of
epidemics. Part I},
\newblock Proc. R. Soc. Lond. A, 115 (1927), 700--721.
\bibitem{K} \newblock R. Z. Khasiminsky,
\newblock \emph{Stability of Systems of Differential Equations under Random Perturbation of Their Parameters},
\newblock in Russian, Nauka, Moscow, 1969.
\bibitem{LJ} \newblock  Y. Lin and D. Jiang,
\newblock \emph{Long-time behaviour of a perturbed SIR model by white noise},
\newblock  DCDS-B, 18 (2013), 1873--1887.

\bibitem{LM} \newblock A. Lasota and M. C. Mackey,
\newblock \emph{Chaos, Fractals, and Noise: Stochastic Aspects of Dynamics}, 2nd ed.,
\newblock Appl. Math. Sci. 97, Springer, New York, 1994.

\bibitem{Lu} \newblock Q. Lu,
\newblock \emph{Stability of SIRS system with random perturbations},
\newblock Phys. A, 388 (2009), 3677--3686.


\bibitem{MLL} \newblock M. C. Mackey, A. Longtin, A. Lasota,
\newblock \emph{Noise-induced global asymptotic stability},
\newblock J. Stat. Phys., 60 (1990), 735--751.

\bibitem{TBV} \newblock  E. Tornatore, S. M. Buccellato, and P. Vetro,
\newblock  \emph{Stability of a stochastic SIR system},
\newblock Phys. A, 354 (2005), 111--126.

\bibitem{VW} \newblock P. van den Driessche and J. Watmough,
\newblock \emph{Reproduction numbers and sub-threshold endemic equilibria for compartmental models of disease transmission},
\newblock Math. Biosci., 180 (2002), 29--48.


\end{thebibliography}
\end{document}